\documentclass{article}

\usepackage[scaled=0.90]{helvet} 
\usepackage[english]{babel}
\usepackage[T1]{fontenc}
\usepackage{ucs}
\usepackage[utf8x]{inputenc}
\usepackage{amsmath}
\usepackage{graphicx}
\usepackage{geometry}

\usepackage{amssymb}
\usepackage{mathtools}
\usepackage{scalerel}
\usepackage{tikz-cd}
\usepackage{amsthm}
\usetikzlibrary{quotes,babel,angles}

\usepackage{url}
\usepackage{hyperref}
\usepackage[cmtip,all]{xy}

\newtheorem{theorem}{theorem}[section]
\newtheorem{lemma}[theorem]{Lemma}
\newtheorem{prp}[theorem]{Proposition}
\newtheorem{thm}[theorem]{Theorem}
\newtheorem{conj}[theorem]{Conjecture}

\theoremstyle{definition}
\newtheorem{dfn}[theorem]{Definition}
\newtheorem{xmp}[theorem]{Example}

\begin{document}



\author{Alexander Prähauser}
\title{The Category of Twisted Extensions of a Vertex Operator Algebra and its Cohomology}

\date{}

\maketitle

\begin{abstract}
The monoidal category of twisted modules of a Vertex Operator Algebra $V$ is defined and reduced to its 2-group of invertible objects $G_\alpha$, which can be described by a 3-cocycle $\alpha$ on its 0-truncation $G$ with values in the group of units $A$ of the field of definition of $V$ serving as its associator. This cocycle also presents the classifying morphism of an $\infty$-group extension of $G$ by the delooping $BA$. Motivated by this, it is proven that the $\infty$-group extension classified by a 3-cocycle $\alpha$ is presented by the skeletal 2-group $G_\alpha$ with associator $\alpha$. The results are discussed in light of current developments in Moonshine and $(\infty,1)$-topos theory.
\end{abstract}

\tableofcontents

\section{Acknowledgements}
I would like to thank my parents, my supervisor Alexander Bors, my second examiner Goulnara Arzhantseva and my friend Daniel Scherl, who helped me several times at a crucial stage of the proof.

\section{Terminology}
The terminology in this work makes a conscious effort to stay simplistic and logical. Composition is denoted in diagrammatic order, so that the composition of two morphisms $f:A\rightarrow B,g: B\rightarrow C$ is given by $f;g$ instead of $g\circ f$. Names of structures are generally given in simple capital letters, even if these structures inhabit a higher-categorical level. The reasoning behind this is that there is no possibility yet to indicate the categorical level uniformly through font, since this would require a family of fonts $\mathcal{F}_n$ such that the level $n$ is discernable from $\mathcal{F}_n$. Terminology of categorical constructions is generally taken from the nlab \cite{Nlab} (as are the definitions, at least in rough terms). Thus, what is called a strong monoidal functor here is called a weak monoidal functor in some other places and not to be confused with a strict monoidal functor, which is sometimes called strong. Lax and colax variants of the definitions are not mentioned.

\section{Introduction}
Given a group $G$ with a normal subgroup $H$, $G$ can be expressed as an extension of $G/H$ by $H$. Thus, the study of finite groups can be broken down to the study of finite simple groups and their cohomology. This motivated the classification of finite simple groups, the largest collaborative project in pure mathematics to date, which was finally concluded in 2004 with the following result: \\

\begin{thm}[{\cite{FSG}}]
Every finite simple group is isomorphic to either
\begin{enumerate}
    \item a cyclic group,
    \item an alternating group of degree $\geqslant$ 5,
    \item a group of Lie type
\end{enumerate}
or one of 27 sporadic groups\footnote{Where the Tits group is included.}. 
\end{thm}

While this result is very satisfying in some regards, it opens up new questions. In particular, the appearance of the 27 sporadic groups is mystifying, as they do not seem to belong to some larger structure and were often only constructed for the classification process. Are they just instances of the law of small numbers\footnote{This observation by Richard K. Guy states that ``There aren't enough small numbers to meet the many demands made of them.''. A consequence of this law is that exceptional structures get rarer as their cardinality increases.}? According to David Corfield, Gelfand didn't think so: \\

``Sporadic simple groups are not groups, they are objects from a still unknown infinite family, some number of which happened to be groups, just by chance.'' \cite{Cor} \\

Whether this is true or not, ample evidence has mounted that the sporadic groups are essential structures in their own right. For one, they are interrelated in a variety of ways. In particular, all but seven (including the Tits group) of them are subquotients of the largest sporadic group, the monster group $M$ (the happy family) and all but two of them have orders divisible only by primes appearing in the order of the monster group. Even more strikingly, a mysterious connection was found between the monster group (and subsequently some of its subquotients) and modular functions. This connection was seen as so mysterious it has been named $Moonshine$ and it took much construction work to even provide an exact formulation for it. Nowadays, Moonshine is formally expressed using Vertex Operator Algebras (henceforth abbreviated VOAs). In the next section, we define these sophisticated algebraic structures and explain some of the bare fundamentals of Moonshine. After that, we shift our attention to the categories of their representations. These turn out to be reducible to the 2-group of their simple elements, which can be described through the group of automorphisms of the VOA and a 3-cocycle on it. Finally, we provide a more conceptual description of this cocycle by showing it to be a representative of the classifying morphism of the 2-group in Theorem \ref{main}. This is the main result of the current text and applicable to any 2-group that is described by an associator on a group. We also discuss some consequences and provide definitions of the category-theoretic notions we are using in the appendix.

\section{Vertex Operator Algebras and Moonshine}

Several definitions of a VOA can be given. We give the original definition, which, though unenlightning in some regards, allows us to see the action of its automorphism group most directly.

\begin{dfn}
A \emph{(complex) vertex operator algebra} $V$ is a $\mathbb{Z}$-graded vector space $\sum_{n\in \mathbb{Z}} V_n$ with an assignment $Y\colon  V\rightarrow End(V)[[z,z^{-1}]]:u\mapsto Y(u, z)=\sum_{n\in \mathbb{Z}} u_{(n)} z^{-n-1}$ from $V$ to the ring of Laurent series of endomorphisms of $V$, and distinguished elements $1\in V_0$, $\omega\in V_2$ such that

\begin{itemize}
    \item Each $V_n$ is finite-dimensional,
    \item For each $u\in V_k$, $u_{(n)}$ is a linear map from $V_l$ to $V_{k+l-n-1}$;
    \item $Y(1, z)$ is the identity,
    \item $Y(u, 0)1=u$,
    \item  the \emph{Jacobi identity} holds
    \begin{align*}
        &z_0^{-1}\delta\left(\frac{z_1-z_2}{z_0}\right)  Y(u, z_1) Y(v, z_2) - z_0^{-1}\delta\left(\frac{z_2-z_1}{-z_0}\right)Y(v, z_2) Y(u, z_1 ) = \\
        &z_2^{-1}\delta\left(\frac{z_1-z_0}{z_2}\right)Y(Y(u, z_0)v, z_2),
    \end{align*}
    \item For all $u,v$ in $V$ there is an $N=N(u, v)$ such that $u_{(n)}v=0$ for all $n\geq N$,
    \item The operators $L_n=\omega_{(n+1)}$ span a copy of the Virasoro algebra whose central term acts on $V$ as a scalar multiple $c\cdot id_V$,
    \item $L_0 v = nv$ for $v\in V_n$,
    \item $Y(L_{-1} v, z)=\partial_z Y(v,z)$,
    
\end{itemize}

A \emph{morphism of VOAs} is given by a linear map $\phi:V\rightarrow W$ such that

$$\phi(Y(u,z)v)=Y(\alpha(u),z)\alpha(v)$$

as an equality of power series, and $\phi(1)=1$ and $\phi(\omega)=\omega$.
\end{dfn}

Of central importance to us is the interaction between a VOA and its automorphism group. In particular, we can give spectral decompositions for VOAs:

\begin{prp}\label{4.2}
Given a locally finite automorphism $g\in Aut(V)$ of finite order $N$, $V$ can be decomposed into eigenspaces $V^j=\{v\in V \mid gv=\zeta^{j}_N v\}$, where $\zeta_N$ is the $N$th root of unity.
\end{prp}

\begin{proof}
Since $g$ is of order $N$, the subgroup generated by it is isomorphic to the cyclic group $C_N$ of order $N$. Thus its eigenvalues have to be from the group of $N$-th square roots.
\end{proof}

We are mostly interested in the modules of VOAs:

\begin{dfn} \cite{C2}\cite{MBM}
A \emph{weak module} of a VOA $V$ is given by a vector space $M$ equipped with a linear map 

$$Y_M:V\rightarrow End(M)[[z,z^{-1}]]$$
$$v\mapsto Y_M(v,z)=\sum_{n\in\mathbb{Z}} v_nz^{-n-1}, v_n\in End(M)$$

such that

\begin{enumerate}
    \item $v_n m=0$ for all $n\geq n_0(m)$,
    \item $Y_M(1,z)=id_M$,
    \item The Jacobi identity holds:
    \begin{align*}
           &z^{-1}\delta\left(\frac{z_1-z_2}{z_0}\right)Y_M(u,z_1)Y_M(v,z_2) - z_0^{-1}\delta\left(\frac{z_2-z_1}{-z_0}\right)Y_M(v,z_2)Y(u,z_1)= \\
           &z_2^{-1}\delta\left(\frac{z_1-z_0}{z_2}\right) Y_M(Y(u,z_2)v,z_2).
    \end{align*}

\end{enumerate}
A \emph{weak $g$-twisted module} for an automorphism $g$ of order $N$ is given by a vector space $M$ equipped with a linear map

$$Y_M:V\rightarrow End(M)[[z^{\pm\frac{1}{N}}]]$$
$$v\mapsto Y_M(v,z)=\sum_{n\in\frac{\mathbb{Z}}{N}} v_nz^{-n-1}, v_n\in End(M)$$

such that 1., 2. and the \emph{twisted Jacobi identity} hold, for $u\in V^j$ (as in \ref{4.2}): 
\begin{align*}
    &z^{-1}\delta\left(\frac{z_1-z_2}{z_0}\right)Y_M(u,z_1)Y_M(v,z_2)-z_0^{-1}\delta\left(\frac{z_2-z_1}{-z_0}\right)Y_M(v,z_2)Y(u,z_1)= \\ &z_2^{-1}\left(\frac{z_1-z_0}{z_2}\right)^{-j/N}\delta\left(\frac{z_1-z_0}{z_2}\right) Y_M(Y(u,z_2)v,z_2)
\end{align*}

A weak $V$-module is \emph{admissible} if it carries a $(\mathbb{Z},+)$-grading $M=\bigoplus_{n\in \mathbb{Z}} M_n$ such that, if $v\in V_r$, then $v_mM_n\subseteq M_{n+r-m-1}$. \\

A weak twisted $V$-module is \emph{admissible} if it carries a $\frac{\mathbb{Z}}{N}$-grading $M=\bigoplus_{n\in \frac{\mathbb{Z}}{N}} M_n$ such that, if $v\in V_r$, then $v_mM_n\subseteq M_{n+r-m-1}$. \\

A weak (twisted) $V$-module is \emph{ordinary} if it carries a $\mathbb{C}$-grading $M=\bigoplus_{\lambda\in\mathbb{C}} M_\lambda$ such that
\begin{itemize}
    \item $dim(M_\lambda)<\infty$,
    \item $M_{\lambda+n}=0$ for fixed $\lambda$ and $n<<0$,
    \item $L_0 m = \omega_1(m) = \lambda m\ \forall m$.
\end{itemize}

A weak (twisted) $V$-module is \emph{irreducible} if it is not a direct sum of other weak (twisted) modules.
\end{dfn} 

It might not be obvious that an admissible module is ordinary, but it follows from the following proposition:

\begin{prp}
For an irreducible admissible module $M$ of a VOA $V$ there exists an $h\in\mathbb{Q}$ such that $M_h$ is nonzero and if $M_\alpha\neq 0$ for some $\alpha\in\mathbb{C}$, then $\alpha-h\in\mathbb{N}$.
\end{prp}

\begin{proof}
See page 244 of \cite{Zhu}.
\end{proof}

\begin{dfn}
The \emph{conformal weight} $h(M)$ of an admissible module $M$ is defined as the smallest $L_0$-eigenvalue on $M$.
\end{dfn}

We require some regularity assumptions on our VOAs:

\begin{dfn}\label{4.6}
A VOA $V$ is \emph{of CFT type} if $V_{k}$ is trivial for negative $k$ and $V_0$ is spanned by $1$. \\

A VOA $V$ is \emph{$C_2$-cofinite} if $V/C_2(V)$ is finite-dimensional, where $C_2(V)=\{u_{(-2)}v\mid u,v\in V\}$. \\

A VOA $V$ is \emph{regular} if every weak $V$-module is a direct sum of simple weak $V$-modules. \\

A VOA $V$ is \emph{weakly rational} if it is regular, has only a finite number of irreducible weak $V$-modules and every irreducible weak $V$-module is an ordinary $V$-module. \\

A VOA $V$ is \emph{holomorphic} if it is regular and has a unique simple module.
\end{dfn}

The connection of VOAs to Moonshine  comes about through the graded dimensions of their characters:

\begin{dfn}
Given an element $v\in V_n$, the \emph{zero-mode} $o(v)$ of $v$ is given by $v_{(n-1)}$.
\end{dfn}

\begin{dfn}
Given a weak module $M$ of a VOA $V$, its \emph{character} is defined as

$$\chi_M(\tau,v):=q^{-\frac{c}{24}}\sum_{n=0}^{\infty} tr_{M_{h(M)+n}} o(v) q^{h+n}$$

where $q=e^{2\pi i \tau}$, $h$ is the height of $M$, $\tau\in \mathbb{C}$ and $v\in V$. 
\end{dfn}

\begin{thm}
If $V$ is a $C_2$-cofinite weakly rational VOA with $\Phi(V)$ its set of irreducible modules, then there is a representation $\rho$ of $SL_2(\mathbb{Z})$ by complex matrices $\rho(A)$ indexed by $V$-modules $M,N\in\Phi(V)$, such that the characters of the modules of $V$ obey the relation

$$\chi_M\left(\frac{a\tau+b}{c\tau+d},v\right)=(c\tau+d)^n\sum_{N\in\Phi(V)}\rho\begin{pmatrix}
a & b \\
c & d
\end{pmatrix}_{M, N}
\chi_N(\tau,v)
$$
\end{thm}

This is \emph{Zhu's Theorem} \cite{Zhu}, the fundamental result in the formalization of Moonshine. Given the necessary vocabulary, this theorem can be summarized by saying that $\chi_M$ is of weight $n$ and multiplier $\rho$ (see \cite{MBM}). It is hoped that VOAs can be used for the classification of finite simple groups in a similar way to how Lie algebras can be used for the classification of simple Lie groups \cite{F}.

\begin{xmp}
(\cite{MBM}, introduction) The graded dimension $\chi_M(\tau,1)$ of the Monster VOA is the $SL_2(\mathbb{C})$-hauptmodul $J$.
\end{xmp}

Since the category of modules of a holomorphic VOA is semisimple linear with one simple object, it is equivalent to $Vect_\mathbb{C}$. We are more interested in the category of twisted modules of $V$. For a particular $g\in Aut(V)$, the category $C_g$ of $g$-twisted modules is again equivalent to $Vect_\mathbb{C}$ \cite{DLM}. However, the category of twisted modules $C=\bigoplus_{g\in G} C_g$ exhibits a nontrivial monoidal product. For a description of this product, we assume the regularity conjecture holds:

\begin{conj}
Let $V$ be a holomorphic VOA and and $G\subseteq Aut(V)$ be a finite group. Then the sub-algebra of fixed points of $V$ under $G$ is regular as in Definition \ref{4.6}.
\end{conj}

The regularity conjecture is widely assumed to hold and was proven for solvable groups \cite{GOG, CM}. With it in place, we can describe the monoidal product of twisted modules of a VOA. We do this in the next section.

\section{Monoidal Categories}\label{Mono}
Monoidal categories are a categorification of monoids. While examples of their use abound, we will have to go fairly deeply into their particularities, which justifies repeating their definition:

\begin{dfn}
A monoidal category is a tuple $(M,\otimes, \alpha, \lambda, \rho)$ consisting of 
\begin{itemize}

\item a category $M$

\item a functor

   $
      \otimes 
        \;\colon\; 
      M \times M  
       \longrightarrow
      M
   $

\item an object

   $
     1 \in M 
   $

   called the unit object, 

\item a natural isomorphism 

   $
     \alpha
       \;\colon\; 
     (-)\otimes ((-)\otimes (-))
       \overset{\simeq}{\longrightarrow}
     ((-)\otimes (-))\otimes (-)
   $

   with components of the form

   $
     \alpha_{x,y,z} : (x \otimes y)\otimes z \to x\otimes (y\otimes z) 
   $

   called the associator, 

\item a natural isomorphism 

   $
     \lambda 
       \;\colon\; 
     (1 \otimes (-)) 
       \overset{\simeq}{\longrightarrow}
     (-)
   $

   called the \emph{left unitor}, and

\item a natural isomorphism

   $
     \rho \;\colon\; (-) \otimes 1 \overset{\simeq}{\longrightarrow} (-)
   $

    called the \emph{right unitor}, 

\end{itemize}
such that the following two kinds of diagrams commute, for all objects involved:
\begin{itemize}
\item the \emph{pentagon identity}
$$
\begin{tikzcd}[column sep=small]
& ((X\otimes Y)\otimes Z) \otimes W \arrow[r, "\alpha_{X\otimes Y, Z, W}"] \arrow[d, "\alpha_{X Y, Z \otimes id_W}"] & (X\otimes Y)\otimes (Z \otimes W) \arrow[r, "\alpha_{X Y, Z\otimes W}"] & (X\otimes Y(\otimes (Z \otimes W)) \\
& (X\otimes (Y\otimes Z)) \otimes W \arrow[rr, "\alpha_{X,Y\otimes Z, W}"] & & X\otimes ((Y\otimes Z) \otimes W) \arrow[u, "id_X \otimes \alpha_{Y, Z, W}"]
\end{tikzcd}
$$

\item the triangle identity
$$
\begin{tikzcd}[column sep=small]
& (X\otimes 1) \otimes Y  \arrow[rr, "\alpha_{X,1,Y}"] \arrow[dr, "\rho_X\otimes id_Y"] & & X\otimes(1\otimes Y) \arrow[dl,"id_X\lambda_Y"]\\ 
& & X\otimes Y
\end{tikzcd}
$$
\end{itemize}
\end{dfn}

Monoidal categories are a categorification of monoids, so that a monoidal category with trivial 1-structure is just a monoid. 

\begin{xmp}\label{Galpha}
\cite{EGNO} Given any group $G$, field $\mathbb{K}$ and $\mathbb{K}^\times$ valued 3-cocycle $\alpha\in C^3(G, \mathbb{K}^\times)=Grp_\infty(BG, B^3 A)$, we can define a category $G_\alpha$. Its set of objects $(G_\alpha)_0$ is isomorphic to the set $G$, and for objects $g_\alpha\,h_\alpha \in (G_\alpha)_0$, its morphisms are given by $Hom(g_\alpha, h_\alpha) = \delta(g, h) End(\mathbb{K})$, where $\delta$ is the Kronecker delta and $End$ denotes the linear endomorphisms. In particular, the automorphism group of any object corresponds to $\mathbb{K}^\times$. Its monoidal product is given by $g_\alpha\otimes h_\alpha=(gh)_\alpha$, its unit isomorphism is the identity and its associator $\alpha_{-,-,-}$ is given by a 3-cocycle $\alpha(-,-,-)$.
\end{xmp}

$G_\alpha$ is an example of a skeletal monoidal category:

\begin{dfn}
A category is \emph{skeletal} if every isomorphism class contains only one object. A monoidal category is \emph{strict} if the associator and unitor isomorphisms are the identity.
\end{dfn}

\begin{thm}
Every monoidal category is equivalent to a skeletal monoidal category with trivial unitors.
\end{thm}

\begin{proof}
Every monoidal category is equivalent to a skeletal monoidal category where the unitors and associators are turned into automorphisms. By Theorem 3.2 in \cite{Mst}, an equivalent monoidal product can be defined on the same category such that its unitors are the identity.
\end{proof}

Moreover, the examples we were looking at have both the property that every object $g$ has a weak inverse $g^{-1}$ under the monoidal product $\otimes$ in the sense that there is an isomorphism $g\otimes g^{-1}\overset{\simeq}{\longrightarrow} 1$, and that every morphism is invertible. Thus they are not only monoidal categories, they are \emph{2-groups}. So we can see that the difference in specificity between a monoidal category and a group can be understood in terms of the various intermediary stages that are each more special than the notion of a monoidal category and less special than that of a group, but in different ways: a monoidal category $M$ can be reduced to a monoid $M_0$ by \emph{0-truncation}, the trivializing of all categorical structure, and if all objects of the monoidal category are invertible under the monoidal product, that 0-truncation is a group\footnote{More precisely 0-truncation is usually defined in geometrical terms as the filling of all $>0$ cells of the simplicial nerve of the monoidal category (see chapter 7).} \footnote{On the other hand, each monoid can be understood as a monoidal category with trivial morphisms.}. Another intermediary step is the reduction of the underlying category to its \emph{core groupoid}, discarding all noninvertible morphisms. Checking the definition, it can readily be seen that the monoidal product preserves invertible structure, and can thus be restricted to the core groupoid. If we then also restrict our class of objects to those that are invertible under the monoidal product, we obtain the \emph{Picard 2-group} $Pic(M)$ of our monoidal category $M$. This mapping is functorial. Applying then 0-truncation, or alternatively, extracting the maximal subgroup of the 0-truncation of our monoidal category, we arrive at a group as the last step in this reduction. \\

One advantage of restricting to 2-groups is that for those, through cohomology, the procedure of 0-truncation is invertible:

\begin{thm}\label{Ass}
Given a skeletal 2-group $G$, the group of automorphisms $A$ of each object is the same and abelian. Furthermore, if its truncation is the group $G_0$, then the two have the same set of objects and the associator $\alpha_{-,-,-}$ of $G$ is a 3-cocycle $\alpha(-,-,-)$ on $G_0$ with values in the abelian group $A$. Furthermore, 2-groups with truncation $G_0$ are in correspondence with cohomology classes on $G$ modulo the action of $Out(G_0)$ on the sets of cohomology classes on $G_0$. In particular, the data of a 2-group $G_\alpha$ is equivalent to that of a 3-cocycle $\alpha$ on its 0-truncation. 
\end{thm}

\begin{proof}
This is Theorem 2.11.5 in \cite{EGNO}.
\end{proof}

Monoidal categories are the canonical categorification of monoids. The archetypical example of a monoid is the monoid of a ring. The categorification of a ring, or, more generally, an $R$-algebra, is slightly less canonical. One possible approach is using fusion categories.

\begin{dfn}
An object $X$ in a monoidal category is \emph{right dualizable} if there exists an object $X^\ast$ and morphisms $ev_X:X^\ast\otimes X\rightarrow 1$ and $coev_X:1\rightarrow X\otimes X^{\ast}$ such that 

$$
\begin{tikzcd}[column sep=huge]
X^\ast \otimes (X\otimes X^\ast) \arrow[d, "\alpha_{X^\ast,X,X^\ast}^{-1}"] & X^\ast\otimes 1 \arrow[l, "id_{X^\ast}\otimes coev_X"] \arrow[d, "\rho_{X^{\ast}}\lambda_{X^\ast}^{-1}"] \\
(X^\ast\otimes X)\otimes X^\ast \arrow[r, "ev_X\otimes id_{X^\ast}"] & 1\otimes X^\ast \\
\end{tikzcd}
$$

and

$$
\begin{tikzcd}[column sep=huge]
(X\otimes X^\ast)\otimes X \arrow[d, "\alpha_{X,X^\ast,X}"] & 1\otimes X \arrow[l, "coev_X \otimes id_{X^\ast}"] \arrow[d, "\lambda_{X};\rho_X^{-1}"] \\
X\otimes (X^\ast\otimes X) \arrow[r, "id_{X^\ast}\otimes ev_X "] & X\otimes 1 \\
\end{tikzcd}
$$

An object $X$ is \emph{left dualizable} if there exists an object $X_\ast$ with the same morphisms and conditions as above with the positions of $X$ and $X_\ast$ reversed. \\

A monoidal category $C$ is \emph{rigid} if every object is left and right dualizable. \\

A monoidal category $C$ is \emph{semisimple} if every object in it is a direct sum of finite simple objects. \\

A monoidal category $C$ is \emph{linear} over a ground ring $R$ if each $hom$-set in $C$ has naturally the structure of an $R$-module and the composition and identity morphisms are bilinear\footnote{In other words, it is \emph{enriched} over $R-mod$. We will not go into the theory of enriched categories, but see, for instance, the nlab page \url{http://nlab-pages.s3.us-east-2.amazonaws.com/nlab/show/enriched+category}.}. \\

A \emph{fusion category} $A$ is a rigid semisimple linear monoidal category with only finitely many isomorphism classes of simple objects, such that the endomorphism monoid of the unit object is isomorphic to the multiplicative monoid of the ground ring $R$\footnote{Normally the ring is assumed to be a field, however, this is not necessary.}. \\

A fusion category is \emph{pointed} if all of its simple objects are invertible under the monoidal product. It is called \emph{G-pointed} if additionally no non-trivial morphisms exist between non-isomorphic simple objects and $G$ is the 0-truncation of its subcategory of invertible objects.\footnote{This terminology is unrelated to \ref{pointed}.}
\end{dfn}

The theory of pointed fusion categories is very similar to that of $R$-algebras with addition corresponding to the direct sum, multiplication to the tensor product, simple objects to irreducibles and units to invertible objects\footnote{With one important difference being that pointed fusion categories do not have additive inverses.}. 

\begin{xmp}
$\cite{EGNO}$ If $A$ is the group of units of a field $\mathbb{K}$, a 3-cocycle $G^3\rightarrow A$ generates a $G$-pointed fusion category $Vect_\alpha$. The objects of this category are $G$-graded vector spaces $\sum_{g\in G} \mathbb{K}_g^{n_g}$, the morphisms are given by componentwise vector space morphisms, the multiplication is given by $\mathbb{K}_g\otimes \mathbb{K}_h = \mathbb{K}_{gh}$ and linearly extended, and the associator is given by $\alpha$ on the primitive objects $\mathbb{K}_g$ and linearly extended. The Picard 2-group of this category is $G_\alpha$ from example \ref{Galpha}.
\end{xmp}

$G_\alpha$ is a good reduction of $Vect_\alpha$ in the sense that the restriction of the 2-Picard functor from the 2-category of monoidal categories and functors to the category of 2-groups is fully faithful when restricted to the non-full subcategory of pointed semisimple distributive monoidal categories and monoidal functors preserving direct sums.

\begin{prp}
\cite{Kri02}(\cite{GOG} Section 2.3.) Given the category $C$ of twisted modules of a holomorphic VOA $V$ (for which the regularity conjecture holds) with simple objects $T_g$, there is a $\mathbb{C}^{\times}$-valued cocycle $\alpha$ on the automorphism group $G$ of $V$ such that the biproduct-preserving functor $C\rightarrow Vect_\alpha$ given by $T_g\mapsto g_\alpha$ on simple objects is an equivalence that preserves the monoidal product. 
\end{prp}

\begin{proof}
The tensor product on $C$ is described in \cite{Kri02}, and it follows from the results there that $C$ is a $G$-pointed fusion category. This can be made into a skeletal fusion category with trivial unitors. The invertible objects of this category form a skeletal 2-group whose associator is, by theorem \ref{Ass} given by a 3-cocycle $\alpha$ on its truncation $G$. Since $C$ is fusion, thus in particular distributive, this fixes the tensor product on $C$.
\end{proof}
    
This powerful result gives us a clear idea of the structure of $C$. In particular, we obtain a monoidal embedding $\iota:G_\alpha\hookrightarrow C:g_\alpha\mapsto T_g$ corresponding to the embedding of the group of units of a ring.

\section{Deloopings and $(n,m)$-categories}
Given any monoidal category, we can define its delooping:

\begin{dfn}
The delooping of a pointed object $M$, in the sense of definition \ref{pointed}, in an $(n,m)$-category with a terminal object is a connected object $B M$, such that the pullback of the embedding $*\rightarrow B M$ of the terminal object along itself is equivalent to $M$ in the sense of higher category theory.
\end{dfn}

$(n,m)$ is to be understood here as in higher category theory, i.e. a category with morphisms of order $m$, such that all morphisms of higher order than $n$ are invertible. $(0,m)$-categories are also called $m$-groupoids. Formally defining an $(n,m)$-category, or even an $m$-groupoid, is very hard, since it requires the formulation of coherence conditions that quickly become exceedingly complex. Thus higher categories are often modeled using (directed) geometric models that contain the coherence conditions in their geometry. We will see one such approach, using simplicial sets to model $\infty$-groupoids, in the next chapter.

\begin{xmp}
The delooping for a group $G$ is a $1$-object groupoid $BG$ with morphisms labeled by the elements $g\in G$ and composition given by the group operation. Please note that the normal definition $f\circ g:=fg$ results in an order reversal when morphisms are written in diagrammatic order.
\end{xmp}

\begin{xmp}
The delooping of a monoidal category $M$ is given by a $1$-object $2$-category $B M$ with $1$-morphisms labeled by objects $A\in M_0$, composition of 1-morphisms given by the monoidal product $A\otimes B$, 2-morphisms given by morphisms $f,g\in M_1$, where horizontal composition for $a,b\in M(f,g)$ is given by composition in $M$ and vertical composition for $a\in M(f,g),b\in M(g,h)$ given by $b;a=a\otimes b$. Note again that the order is reversed in diagrammatic notation. In particular, the associator becomes a morphism $f;(g;h)=(f;g);h$.
\end{xmp}

So the delooping of a $k$-connected monoidal $(n,m)$-category is a $(k+1)$-connected $(n+1,m+1)$-category, where the monoidal product is translated into composition of 1-morphisms. In particular, the delooping of a monoidal $(n_1,m_1)$-category can only exist in an $(n_2,m_2)$-category if $n_2>n_1+1$ and $m_2>m_1+1$. For instance, the delooping of a group is a groupoid, and as such part of the $(1,2)$-category of groupoids, while the delooping of a monoidal category is a 2-category and as such part of the 3-category of 2-categories. However, 3-categories require far more machinery than 2-categories. Our goal is to instead use $\infty$-groupoids.

\begin{prp}\label{EHprp}
The delooping of a group $G$ admits a monoidal structure if and only if the group is abelian.
\end{prp}

\begin{proof}
Given a monoidal structure on the delooping $BG$, then the isomorphisms on the unique object can be multiplied both by composition and the monoidal product. Thus it follows from the Eckmann-Hilton argument (Theorem 5.4.2 in \cite{EH}) that $G$ is abelian.
\end{proof}

Now, given a 2-group $G_\alpha$ with associator $\alpha:G\rightarrow A$, the automorphism group of each object of $G_\alpha$ is isomorphic to $A$, as noted in \ref{Ass}. Thus we obtain a monoidal embedding $\kappa:B A\rightarrow G_\alpha$ given by mapping the unique object of $B A$ to the identity of $G_\alpha$ and the morphisms of $BA$, which are labeled by the objects of $A$, to the corresponding morphisms of $1_\alpha$. Taking the delooping $B\kappa$ of this embedding, we see that it is equivalent to the embedding of the fiber of the 1-truncation of $BG_\alpha$, and that the image of this truncation $\tau_1 G_\alpha$ is equivalent to $BG$. It then follows from the universal property of truncations that the sequence $B A\rightarrow G_\alpha\rightarrow G$ is a short exact sequence of 2-groups. This is actually a special instance of the homotopy fiber sequences of homotopy types which, by the homotopy hypothesis, are equivalent to $\infty$-groupoids. \\

From the structure of deloopings we can readily see that the generalization from groups to monoids corresponds to the generalization from groupoids to categories. Of course, groups are much less complex than monoids. This discrepancy deepens as the level of categorification is increased. In particular, the theory of 3-categories requires often page-sized diagrams and there are serious hurdles to developing a $3$-topos theory, while the theory of $\infty$-groupoids can be understood as the internal theory of $\infty-Gpd$ (and more refined structures, such as smooth $\infty$-groupoids can be handled in the context of $(\infty,1)$-topos theory). Thus by reducing the twisted representation theory of VOAs to their $2$-group of simple twisted modules, we can use the tools of $\infty-Gpd$, which has a well-developed cohomology theory.

\section{Simplicial Sets}
Homotopy theory replaces a topological space by its homotopy type, which can be defined as its equivalence class under homotopy equivalence. However, for this definition to work on a categorical level, it has to be extended to morphisms, which cannot be done in the most straightforward way. The reason for this is that the morphisms on a topological space are not generally dependent on the homotopy class of that space, not even if taken up to homeomorphism. The resulting categorical apparatus culminates in $(\infty,1)$-topos theory. To be able to effectively calculate with the notions of this theory back to traditional mathematics, first the information of a topological space has to be reduced into a more combinatorial form. We will do this using simplices.

\begin{dfn}
The \emph{simplex category} $\Delta$ is the category consisting of the (unique up to unique isomorphism) nonempty finite total orders $\Delta_n$ on $n+1$ elements, and order-preserving functions. 
\end{dfn}

The motivation of this definition is completely geometric: the total order $\Delta_n$ is to be thought of as the $n$-simplex with its vertices numbered from $0$ to $n$. Any morphism in $\Delta$ can be decomposed into the \emph{coface maps} $d_i:\Delta_{n-1}\rightarrow \Delta_n$ for ${0\leq i \leq n}$, uniquely mapping the $(n-1)$-simplex to the face of the $n$-simplex that is opposite to its $i$-th vertex, and the \emph{codegeneracy maps} $p_{0\leq i \leq n}:\Delta_{n+1}\rightarrow \Delta_n$, uniquely mapping the $(n+1)$-simplex to the $n$ simplex by identifying the $i$-th and $(i+1)$-th vertex. Simplices are extremely useful by allowing us to model a variety of things. They mainly do this through simplicial sets.

\begin{dfn}
A \emph{simplicial set} is a contravariant functor $S:\Delta^{op}\rightarrow Set$ from the simplex category to the category of sets. More generally, a \emph{simplicial object} is a contravariant functor on the simplex category.
\end{dfn}

Simplicial sets are the presheaves on the simplex category, and as such follow the general logic of presheaves as generalized spaces: due to the Yoneda lemma, the set $S(\Delta_n)$ is equivalent to the set $Hom_{Set^{\Delta{op}}}(\Delta_n, S)$, where $\Delta_n$ is interpreted as a simplicial set through the Yoneda embedding and the equivalence is due to the Yoneda lemma, and should be thought of as the $n$-simplices contained in $S$, and the face and degeneracy maps $\delta_i:=S(d_i)\cong Hom_{Set^{\Delta{op}}}(d_i, S)$ and $\pi_i:=S(p_i)\cong Hom_{Set^{\Delta{op}}}(p_i, S)$ , given by precomposition with the coface and codegeneracy maps, map the set of $n$-simplices to their $i$-th faces and the degenerated $(n+1)$-simplices built by taking the $i$-th vertex twice. This description can be formalized by replacing every $n$-simplex of a simplicial set $S$, given as an object of the set $S(\Delta_n)$, with the topological $n$ simplex, with the face and degeneracy maps given in the straightforward way. This is known as the \emph{geometric realization} $|S|$ of $S$. The process can also be turned around, and topological spaces can be made into simplicial sets:

\begin{dfn}
Let $\Delta_n$ be the topological $n$-simplex, and let the simplex category $\Delta$ be embedded in the category of topological spaces by the functor $\Delta_n\rightarrow\Delta_n$ with obvious face and degeneracy maps. Then, for every topological space $X$, let the $\emph{singular simplicial complex}$ $\Delta_\bullet(X)$ be the simplicial set given in degree $n$ by $Hom_{Top}(\Delta_n, X)$, with face and degeneracy maps given by precomposition with the coface and codegeneracy maps of $\Delta$.
\end{dfn}

So $\Delta_\bullet(X)$ is the simplicial set obtained by filling up $X$ with simplices, so that it models $X$, and its geometric realization consists just of the (topological) simplices in $X$, glued wherever their face maps map to the same object. Most importantly, it is a model for the homotopy type of $X$. To see this, we have to introduce some elementary definitions of simplicial sets. Note that these definitions implicitly use that the presheaf category $Set^{\Delta^{op}}$ affords the complete internal logic of a topos, so that in particular unions and complements of subobjects can be formed.

\begin{dfn}
The \emph{boundary} $\partial \Delta_n$ of $\Delta_n$ is formed by the union of the $n$ face subsimplices of dimension $n-1$. The \emph{i-th horn} $\Lambda^i_n$ of $\Delta_n$ is then obtained from $\partial\Delta_n$ by removing the $i$-th face, or equivalently, the union of all face subsimplices except the $i$-th one.
\end{dfn}

\begin{thm}
Given a topological space $X$, the geometric realization of the singular simplicial complex of $X$ is homotopy-equivalent to $X$.
\end{thm}

\begin{proof}
The geometric realization of $\partial\Delta_n$ is homotopy-equivalent to the $n$-sphere. Thus $X$ and $|\Delta_\bullet (X)|$ have the same homotopy groups, thus they are homotopy equivalent by the Whitehead theorem.
\end{proof}

To see why simplicial sets are a convenient model for calculations with $\infty$-groupoids, we have to introduce the notion of a Kan complex.

\begin{dfn}
A \emph{Kan complex} $K$ is a simplicial set fulfilling the \emph{horn filling condition}: any horn embedding $\Lambda^i_n\hookrightarrow K$ can be extended to a simplex embedding $\Delta_n\hookrightarrow K$.
\end{dfn}
Kan complexes give a combinatorial model of both $\infty$-groupoids and homotopy types. We have seen how they model homotopy types using the singular simplicial complex, to see how they model $\infty$-groupoids, we will use a similar construction. Every order induces a category with the same objects and a morphism from $a$ to $b$ iff $a\leq b$. Thus the simplical category $\Delta$ is a subcategory of the $2$-category of categories $Cat$ and therefore also of the $3$-category of 2-categories $Cat_2$. 

\begin{dfn}
The \emph{Duskin nerve} $N(C)$ of a $2$-category $C$ is the simplicial set given in degree $n$ by $Hom_{Cat_2}(\Delta_n, C)$, with face and degeneracy maps given by precomposition.
\end{dfn}

\begin{xmp}\label{Dusk}
Using proposition 5.4.12 of \cite{2N}, we see that the $0$-simplices of $N(C)$ are simply the points of $C$, the $1$-simplices are morphisms $f$, the 2-simplices are natural transformations $\theta:f;g\rightarrow h$, where $f$ and $g$ form the zeroth and second face of the $2$-simplex and $h$ the first, $3$-simplices are given by quadruples of commuting natural transformations $\theta_0,\theta_1,\theta_2,\theta_3$ and all higher structure is induced. The case we are interested in is that of a skeletal semigroup with strict units $G_\alpha$, so that any natural transformation is an automorphism $a\in A$, so that any 2-simplex has the form $a:f;g\rightarrow f;g$ and any 3-simplex $\Delta_3$ has edges $\Delta^{01}_3=f,\Delta^{12}_3=g,\Delta^{23}_3=h,\Delta^{02}_3=f;g,\Delta^{13}_3=g;h,\Delta^{03}_3=f;g;h$ and the commutation condition can be written as an equation $\theta_0\theta_1=\alpha\theta_2\theta_3$. \\
\end{xmp}

The nerve of a 2-category $C$ is generally not a Kan complex, but it is if $C$ is a 2-group. Then the horn filling condition, restricted to the inner horn $\Lambda^1_2\hookrightarrow K$ of the $2$-simplex, is an equivalent to the categorical condition that every pair of morphisms $f,g$ with $src(g)=tar(f)$ can be composed, but without imposing a uniqueness condition on the composition, as any filling of the horn can be interpreted as a composite. On the outer horns $\Lambda^0_2\hookrightarrow K, \Lambda^2_2\hookrightarrow K$, the horn filling condition means that a pair of morphisms $f,g$ with the same source or target can \emph{also} be composed, in the sense that a morphism $h$ exists such that $f;h=g$. In particular, if $g$ is the degenerated $1$-simplex on $src(f)$, which exists thanks to the degeneracy morphisms of a simplicial set, then $h$ is an inverse of $f$. Thus the inner horn lifting conditions of a Kan complex are generalizations of the composability to arbitrary dimensions, and the outer horn lifting conditions are generalizations of the existence condition on inverses, both expressed in a geometric (or combinatorial) language. So the Duskin nerve of a 2-group is a Kan complex, as is the singular simplicial complex of a topological space. Thus Kan complexes inhabit a sweet spot between algebra, combinatorics and topology, and can be used to show the homotopy hypothesis, that an $\infty$-groupoid is the same as a homotopy type. However, to provide a model for homotopy, the category of simplicial sets needs homotopy equivalences. We will import these from the category of topological spaces:

\begin{dfn}
A morphism between simplicial sets is a \emph{weak (Quillen) equivalence} if its geometric realization is a homotopy equivalence.
\end{dfn}

\begin{dfn}
A \emph{category with weak equivalences} is a category $C$ equipped with a set $W$ of morphisms $f\in Mor(C)$, such that all isomorphisms are in $W$ and $W$ fulfills the two-out-of-three condition: for all $f,g\in Mor(C)$, if any two of the three $f,g$ and $f;g$ are in $W$, then the third is in $W$ too.
\end{dfn}

\begin{xmp}
Homotopy equivalences form a category with weak equivalences, as do weak Quillen equivalences.
\end{xmp}

Weak equivalences are additional 1-categorical structure that describes intrinsic $(\infty,1)$-categorical structure. More precisely, each category with weak equivalences $C$ describes a unique $(\infty,1)$-category $\bar{C}$, which can be obtained from $C$ by Hammock localization. Basically, the morphisms between two objects $X,Y\in\bar{C}$ form an $\infty$-groupoid whose objects are sequences of morphisms $(f_1,w_1, f_2, w_2,...,w_{n-1}, f_n)$ in $C$, such that $src(f_1)=X, tar(f_n)=Y$, all $w_i$ are weak equivalences, and $src(f_i)=tar(w_i)$, so basically "morphisms up to weak equivalence". However, since the $(\infty,1)$-category we are trying to define \emph{is} the category of $\infty$-groupoids, we would run into definitional difficulties if we were trying to define it through Hammock localization. We will instead not formally define $\infty-Gpd$, since we will not need it directly, and point the reader to \cite{HTT} for a clean definition. But we want to cite one particular example of $\infty$-topos theory, which serves as motivation for the main result.

\begin{prp} \label{lastprp}
An extension of $\infty$-group objects $K\rightarrow G \rightarrow H$ in an $(\infty, 1)$-topos $\mathcal{T}$ gives rise to a homotopy-exact sequence $K\rightarrow G \rightarrow H \xrightarrow{c} BK \rightarrow BG \rightarrow BH$. If $K$ is abelian\footnote{Or, more generally, braided}, then $BK$ is also an $\infty$-group object and we can extend this homotopy-exact sequence one step to the right to obtain a morphism $BH\xrightarrow{Bc} B^2K$ and the $\infty$-groupoid of $\infty$-group extensions of $H$ by $K$ is equivalent to the $\infty$-groupoid $\mathcal{T}(BH, B^2K)$.
\end{prp}

\begin{proof}
See chapter $5.1.18$ in \cite{dcct}.
\end{proof}

\section{Main Result}
Due to Proposition \ref{lastprp}, the two-fold delooping $B^2K$ of an abelian $\infty$-group $K$ is the classifying space of extensions by $K$ in the sense that every extension of an $\infty$-group $H$ by $K$ is the loop space of the homotopy fiber of a morphism $BH\rightarrow B^2K$. Our goal is to show that the 2-group extension $G_\alpha$ of a group $G$ by a delooped abelian group $BA$ with associator a 3-cocycle $\alpha$\footnote{See the paragraph after the the proof of Proposition \ref{EHprp}.} represents the $\infty$-group extension classified by the morphism $\alpha:BG\rightarrow B^3A$. For this we need to find a presentation of the homotopy fiber of $\alpha$. This boils down to two constructions: first we need to construct a simplicial group representing the $\infty$-group $B^2A$, then we need to construct the simplicial classifying space of $B^2A$ along with its decalage. For the first construction we will start with the chain complex $A[2]$, which we know represents the delooping $B^2A$ in the category of abelian chain complexes. This category is equivalent to the category of simplicial abelian groups via the \emph{Dold-Kan correspondence}. We will not give the general formula of the Dold-Kan nerve $\Gamma$, which can be found in \cite{DK}, but take a detailed look at the case of our chain complex $A[2]$, which is concentrated in degree 2. This nerve is given in simplicial degree $n$ by

$$\Gamma(A[2])_{n}=\bigoplus_{\Delta_n\rightarrow [2]_{surj}} A$$.

 For us, only the first three degrees are relevant. So $\Gamma(A[2])$ is concentrated in degrees $0$ and $1$, consists of one copy of $A$ in degree $2$ and three copies $A_0$, $A_1$, $A_2$ in degree $3$, corresponding to the three degeneracy maps $\sigma_{0}, \sigma_{1}, \sigma_{2}$ between degree 3 and 2. On each factor $A_j$ of the biproduct, the face maps simplify to 

$$\delta_i|_{A_j}\rightarrow A=\begin{cases}
id & \text{if } i=j, j+1 \\
* & \text{otherwise}
\end{cases}$$

Given now any simplicial group $G$, the \emph{simplicial classifying space} $\bar{W}(G)$ can be constructed with its \emph{decalage} $dec:W(G)\rightarrow \bar{W}(G)$, which is a morphism representing the canonical basepoint inclusion of the delooping $BG$, so that the homotopy fiber of the morphism $BG\rightarrow B^3A$, which represents our cocycle $\alpha$, is in turn represented by the pullback of $\alpha$ along $dec$. Again, we do not give the general procedure, which is somewhat tedious, but can be found in Chapter 5 of \cite{Sp} and simplifies in our case. Since degrees 0 and 1 of $\Gamma(A[2])$ are trivial, the general formula for $W(B^2 A)$ simplifies to $W(B^2 A)_{2}\cong\Gamma(A[2])_2=A$ and $W(B^2 A)_3\cong\Gamma(A[2])_3\times \Gamma(A[2])_2\cong A_0\times A_1\times A_2 \times A_3$ for objects of degree 2 and 3. Furthermore, since the face maps from degree 2 to degree 1 of $\Gamma(A[2])$ are all trivial, the face maps between $\bar{W}(B^2 A)_3$ and $\bar{W}(B^2 A)_2$ are given by $\delta_0(a,b,c,d)=\delta_0(a,b,c)d=ad$, $\delta_1(a,b,c,d)=\delta_1(a,b,c)=ab$, $\delta_2(a,b,c,d)=\delta_2(a,b,c)=bc$ and $\delta_3(a,b,c,d)=\delta_3(a,b,c)=c$. $\bar{W}(B^2 A)$ finally is given in degree 2 by the quotient of $A$ with itself, thus trivial, and in degree 3 by $\Gamma(A[2])_3/\Gamma(A[2])_3\times \Gamma(A[2])_2$, thus isomorphic to $A$, and the decalage morphism $dec$ is (equivalent to) the projection onto the fourth factor. Thus, for a 3-cocycle $\alpha:\bar{W}(G)\rightarrow \bar{W}(B^2 A)$, the pullback of $\alpha$ along $dec$ consists of one 0-cell, the objects of $G$ as 1-cells, for each $f,g\in G,a\in A$, 2-cells 
$$
\begin{tikzcd}[column sep=small]
& |[alias=U]|*  \arrow[dr, "f"]  & \\
* \arrow[ur, "g"] \arrow["fg"]{rr}[name=D]{} & & *
\arrow[Rightarrow, from=U, to=D, "a"]
\end{tikzcd}
$$

and, for $\alpha(f,g,h)=d$ and $a,b,c\in A$, 3-cells with edges labeled by $\Delta^{01}_3=f,\Delta^{12}_3=g,\Delta^{23}_3=h,\Delta^{02}_3=f;g,\Delta^{13}_3=g;h,\Delta^{03}_3=f;g;h$ and faces labeled by $\delta_0=ad$, $\delta_1=ab$, $\delta_2=bc$ and $\delta_3=c$. Remembering now the Construction \ref{Dusk} of the Duskin nerve, we see that $N(G_\alpha)$ has the same 0,1 and 2-simplices, and an isomorphism of $3$-simplices can be given by the map $\theta_3\mapsto c, \theta_2\mapsto ab, \theta_1\mapsto bc, \theta_0\mapsto ad$. Plugging in the variables, we see that indeed $\theta_0\theta_1=adbc=\alpha\theta_2\theta_3$. Thus, we have proved

\begin{thm}\label{main}
Given a 3-cocycle $\alpha:G^3\rightarrow A$, the 2-group extension of $G$ by $BA$ with associator $\alpha$ is equivalent to the $\infty$-group extension of $G$ classified by the morphism $BG\rightarrow B^3A$ represented by $\alpha$. 
\end{thm}

\section{Conclusion}
As illustrated in Section \ref{Mono}, the category of twisted modules of a holomorphic VOA (to which the regularity conjecture applies) can be reconstructed from its automorphism group and a $3$-cohomology cocycle, or, equivalently, the $2$-group-extension it classifies. Our main result has made this construction fully explicit in the context of the simplicial model category. Moreover, it has shown that this explicit construction is a presentation of a general result from $(\infty, 1)$-topos theory. It is natural to expect that more of the theory of VOAs and their automorphism groups could find a natural framework in the theory of cohesive $(\infty, 1)$-topos theory. Importing this result already opens up the use of the internal cohomology theory of an $(\infty, 1)$-topos, which is a natural environment for cohomology. \\

Unlike the associator, the classifying morphism it presents is invariant under categorical equivalence and exists in every $(\infty,1)$-topos. This is important since a satisfactory theory of VOAs within the framework of $(\infty,1)$-topoi can be expected to require a refinement of $\infty$-groupoids that allows the formulation of differential cohesion, such as the $(\infty,1)$-topos of formally smooth $\infty$-groupoids. \\

It is not yet understood to what degree the cocycle $\alpha$ is characteristic to its VOA or its automorphism group. This question can be split in two: how many VOAs can have the same automorphism group, and can two different VOAs be Morita equivalent?  The second part of the question is taken up at the end of \cite{GOG}, where it is conjectured that two VOAs are Morita equivalent if and only if they have the same central charge. \\

It is also not understood what the position of the cocycle says about the VOA. For instance, it is known that the cocycle $\alpha_\mathbb{M}$ of the monster module has order 24 and it is conjectured that it generates its cohomology group and that a VOA can only be defined over the integers if its cocycle has order 24 \cite{F}. Thus, information about the cohomology group $H^3 (G, \mathbb{C}^\times)$ of a group $G$ can be used to gain information about VOAs that have it as an automorphism group without having to construct the VOA. This might be particularly useful for existence theorems of $p$-adic VOAs. Given the current interest in defining VOAs over the integers and the appearance of Moonshine phenomena over the adics \cite{adic}, the question of extensions over $p$-adic fields seems like a natural next step. In a sense, the appearance of a $\mathbb{C}^\times$-valued cocycle underlines the particularity of the underlying definitions, as it transfers the theory into the realm of algebra, and algebraic theories in characteristic 0 usually have analogs for positive characteristics. So far, the abstraction from VOAs to their categories of modules seems unhelpful in capturing Moonshine, but studying adic aspects of the theory might provide finer insight. Here again, a transfer to the $(\infty,1)$-categorical setting might be helpful, as it contains Morava $K$-theories, graduations of $\infty$-fields in characteristic $p$ with associated cohomology theories \cite{Mora}. \\

As is often the case with Moonshine, mysterious numerical identities start appearing everywhere. The central charge of the Monster VOA is 24 and so is the order of its cocycle. The weights of other VOAs used in Moonshine are also close to 24 in some way, though those ways differ (for instance the central charge of the shorter Moonshine module is given by 23.5, while the central charge of Duncan's VOA for $Co1$ in \cite{Col} is 12), and it is conjectured, and partially proven, that the Monster is the unique holomorphic VOA with central charge 24 and $dim(V)_1$ \cite{Irr1} \cite{lam2006characterization}. Similarly, the 3-cohomology of most sporadic simple groups for which it was calculated revolves around 24, but of those calculated, only the monster actually has cohomology 24. \\

Most of this is dependent on the regularity conjecture which will hopefully in due time be proven for all groups. More generally, the connection to Moonshine, consisting in the modularity of characters of modules, gets lost when focussing on the category of modules of a VOA alone. To find the categorical formulation of Moonshine, an external formulation of the character of a VOA-module has to be found, which has to include more data than the category of modules.

\appendix
\section{Category-theoretic notions}
We will need some category theory for our results, and will provide here the used definitions and results without providing proofs, which mostly would require additional vocabulary we won't need. As a good introductory book, we recommend Leinster's ``basic Category Theory''\cite{Lein}.

\begin{dfn}\label{cat}
A \emph{category} $C$ consists of two (large) sets\footnote{Category theory often requires classes and sometimes collections that are larger than classes. This issue is usually resolved with Grothendieck universes, set-theoretic universes that are nested in each other.} $C_0,C_1$ of \emph{objects} and \emph{morphisms}, functions $src,tar:C_1\rightarrow C_0$, $id:C_0\rightarrow C_1$ and a function $comp$, which maps each $f,g$ with $tar(f)=src(g)$ to an element $f;g$, such that $src(f;g)=src(f)$, $tar(f;g)=tar(g)$, $(f;g);h=f;(g;h)$, $id(x);f=f$ and $f;id(y)=f$ for all morphisms that can be composed. \\

A \emph{functor} between two categories $F:C\rightarrow D$ consists of a pair of functions $F_0:C_0\rightarrow D_0$, $F_1:C_1\rightarrow D_1$, such that these functions commute with the source, target, identity and composition morphisms. \\
\end{dfn}

This is one of two possible definitions, the other starts with a set of objects and, for each two objects $x,y$, a set of morphisms $hom(x,y)$, then formulates the above conditions in these terms. The two definitions can readily be derived from each other, though their generalizations in higher category theory can be different. An annoying issue with these definitions is that their optimal formulation would use category-theoretic notions. In particular, composition can best be described using pullbacks:

\begin{dfn}
Given a diagram 
$$
\begin{tikzcd}[column sep=small]
 & X  \arrow[d, "f"] \\ 
Y \arrow[r, "g"] & Z
\end{tikzcd}
$$

the \emph{pullback} of the diagram (or of $f$ along $g$ or $g$ along $f$) is a commutative diagram

$$
\begin{tikzcd}[column sep=small]
Z \arrow[r] \arrow[d]  & X  \arrow[d, "f"] \\ 
Y \arrow[r, "g"] & Z
\end{tikzcd}
$$

that is universal in the sense that for every other commutative diagram

$$
\begin{tikzcd}[column sep=small]
A \arrow[r] \arrow[d]  & X  \arrow[d, "f"] \\ 
Y \arrow[r, "g"] & Z
\end{tikzcd}
$$

an arrow $\iota:A\rightarrow Z$ exists such that the composite diagram made of the two rectangles and $\iota$ commutes.
\end{dfn}

\begin{xmp}
In the category $Set$ of sets and functions, a pullback of a function $f:X\rightarrow Y$ along a point $\ast\rightarrow Y$ is the fiber of $f$ over $\ast$. All other pullbacks can be computed pointwise.
\end{xmp}

In particular, $comp$ in Definition \ref{cat} is a function from the pullback of $tar$ along $src$ to $C_1$. \\

Somewhat relatedly, we use pointed objects:

\begin{dfn}\label{pointed}
Given a category $C$ with a terminal object $\ast$, the \emph{category of pointed objects} $C^{\ast \backslash}$ has as objects the arrows $\ast\rightarrow X$ with source the terminal object, and as morphisms commuting triangles

$$
\begin{tikzcd}
X\arrow[r, "f"] & Y \\
\ast \arrow[u] \arrow[ur] \\
\end{tikzcd}
$$
\end{dfn}

Generally, the terminal object of a category are understood as a point. In particular, in the category of sets, the terminal object is the singleton set, and a pointed object in the category of sets is a pointed set. \\

We are mostly interested in monoidal categories or presheaf categories:

\begin{dfn}
Given two functors $F,G:C\Rightarrow D$, a \emph{natural transformation} $\iota:F\Rightarrow G$ is a family of morphisms $\iota_X:F(X)\rightarrow G(X)$, such that, for each $f:X\rightarrow Y$, the diagram

$$
\begin{tikzcd}[column sep=small]
F(X) \arrow[r,"\iota_X"] \arrow[d,"F(f)"] & G(X) \arrow[d,"G(f)"] \\
F(Y) \arrow[r,"\iota_Y"] & G(Y) \\
\end{tikzcd}
$$

commutes. \\

Given any two categories $C,D$, the \emph{functor category} $D^{C}$ is the category of functors between $C$ and $D$ and natural transformations between them.\footnote{Due to size issues, the functor category often lives in a higher Grothendieck universe than $C$ and $D$.} A \emph{natural isomorphism} is a natural transformation in which each component morphism is invertible. \\

Given any category $C$, the \emph{presheaf category} of $C$ is the functor category $Set^{C^{op}}$, where $C^{op}$ is the category with the same objects as $C$ and all morphisms reversed.
\end{dfn}

One of the more subtle issues in category theory is that of equivalence, which is basically ``isomorphism of categories up to isomorphism'':

\begin{dfn}
A \emph{natural isomorphism} $\iota$ is a natural transformation in which each component morphism of $\iota$ is an isomorphism. \\

Two functors $F,G$ are \emph{equivalent} if there is a natural isomorphism between them. \\

An equivalence between two categories $C,D$ is a pair of functors $F:C\Rightarrow D$, $G:D\Rightarrow C$ and a pair of natural isomorphisms $\lambda:F;G\rightarrow id_{C}$, $\rho:id_{D}\rightarrow G;F$.
\end{dfn}

The fundamental importance of presheaf categories comes from the Yoneda lemma.

\begin{dfn}
A presheaf $F:C^{op}\rightarrow Set$ is \emph{representable} if it is equivalent to a functor of the form $hom(-, X)$ for some object $X$ in $C$.
\end{dfn}

The \emph{Yoneda embedding} $hom:C\hookrightarrow Set^{C^{op}}$ is fully faithful. In fact, a lot more is true:

\begin{lemma}[Yoneda lemma]
For every presheaf $F:C^{op}\rightarrow Set$, there is a canonical isomorphism $Hom_{Set^{C^{op}}}hom(-, X)\cong F(X)$
\end{lemma}

The Yoneda lemma has good claim to being the most important result in category theory. At a very basic level it says that the structure of an object $X$ can be determined by through mappings into other objects with the same kind of structure. In the geometric case, where it is most often used, the category $C$ is most often a category of simple geometric objects. The presheaves on $C$ should then be thought of as generalized spaces that are \emph{characterized precisely} by how the objects $X$ of $C$ map into them, where $X$ is identified with its Yoneda embedding $hom_C(-,X)$. We take a closer look at the presheaf category on the simplex category in the main text. Other examples inlcude the category of affine schemes and the cube category. \\

Presheaf categories have many pleasant properties which we lack the vocabulary to list and which make them very similar to the category of sets itself. We will only use the fact that they have cartesian products and direct sums, which are formed componentwise (so $(F+G)(X)=F(X)+G(X)$ and similar for products). \\

This covers the range of 1-category theory we are using. However, we also need some concepts from higher category theory. Similar to how sets can be understood as categories with without non-identity morphisms, categories can be understood as special instances of more sophisticated structures called $n$-categories. We only give some heuristics here for an intuitive understanding of $n$-categories (a rigorous definition for arbitrary $n$ is still a work in progress). An $n$-category $C$ can inductively be understood as a set $C_0$ equipped with, for each two objects $X,Y\in C_0$, an $(n-1)$-category $hom(X,Y)$ and, for each three objects $X,Y,Z$, composition $(n-1)$-functors $hom(X,Y)\times hom(Y,Z)\Rightarrow hom(X,Z)$, and some structure that defines unitarity (that composition of an object $X$ with a unit object is equivalent to $X$) and associativity of composition. 

\begin{xmp}
Categories, functors and natural transformations form a 2-category $Cat$.
\end{xmp}

Defining this rigorously is very hard in particular due to issues of coherence: as generally category-theoretic notions should only hold up to equivalence, and so should the associativity and unitarity of $n$-categories, and similar to how the right notion of equivalence between categories is not isomorphism but ``isomorphism up to isomorphism'', the right notion of equivalence between higher categories becomes more and more involved. Tracing this recursion downwards leaves a net of increasingly obtuse structure, the first taste of which can be seen in the pentagon identity for the associator of a monoidal category. For 1-categories the issue doesn't exist since its $hom$-objects are 0-categories (also known as sets), in which equivalence is equal to identity. For 2-categories a coherence theorem states that each 2-category can be strictified into a 2-category where composition and unitarity hold on the nose. The same however fails already to hold in the case of 3-categories and 4-categories are borderline unworkable if described explicitly. However, there exists a way to greatly decrease coherence issues by first throwing away all non-invertible structure and taking the limit of the iteration: an $\infty$-category is (or should be) a set with $hom$-objects that are themselves $\infty$-categories and an $\infty$-groupoid is an $\infty$-category such that all (higher) morphisms are invertible. Composition and unitarity should then only hold up to $\infty$-categorical equivalence, which itself cannot be reduced to equality. The recursion doesn't have an endpoint anymore. Defining this rigorously might seem even harder, and it is very hard, but the difficulty can be alleviated by modeling the algebraic notions of composition and identity through geometry. In particular, if an equivalence between two objects $x,y$ is understood as a line between the two points $x$ and $y$ and an equivalence between two equivalences $i,j$ between the same two objects $x,y$ as a surface and so on, then the theory of $\infty$-groupoids can be seen to be equivalent to homotopy theory. This is the content of the \emph{homotopy hypothesis}. There are various ways to state the homotopy hypothesis, but the general idea is that the algebraic definition of an $\infty$-groupoid as an $\infty$-category in which all morphisms are invertible is equivalent to the homotopy type of a topological space. It has become less of a hypothesis that has to be proven and more of a criterion for definitions: a correct definition of $\infty$-groupoids is one for which the homotopy hypothesis holds. We present a geometric model for $\infty$-groupoids in Chapter 7. \\

Once $\infty$-groupoids are defined, the equivalences in the definitions of higher categories can be described by taking recurse in the notion of equivalence of $\infty$-groupoids. This program, which was originally devised in Grothendieck's groundbreaking manuscript ``Pursuing Stacks'' \cite{PStacks}, has taken major strides in Lurie's work and is currently developed in category-theoretic circles. However, that undertaking lies outside the scope of the current work.

\addcontentsline{toc}{section}{References}

\bibliographystyle{plainurl}

\bibliography{literatur}

\end{document}